\documentclass[12pt]{article}
\title{An Algorithm to Compute the Topological Euler Characteristic, Chern-Schwartz-MacPherson Class and Segre Class of Projective Varieties}
\author{
        Martin Helmer  \\ \normalsize
        Department of Applied Mathematics,\\ \normalsize
 University of Western Ontario,\\ \normalsize
London, Canada, N6A 5B7\\ \normalsize
  \texttt{martin.helmer2@gmail.com}}
\date{\today}

\usepackage{chngcntr}
\counterwithin{table}{section}
\usepackage{booktabs}% http://ctan.org/pkg/booktabs
\usepackage{colortbl}% http://ctan.org/pkg/colortbl
\usepackage{xcolor}% http://ctan.org/pkg/xcolor
\usepackage{graphicx}% http://ctan.org/pkg/graphicx

 \definecolor{Ftitle}{RGB}{11,46,108}
\definecolor{line}{RGB}{87,39,117}
\colorlet{tableheadcolor}{Ftitle!25} % Table header colour = 25% gray
\colorlet{tablerowcolor}{gray!10} % Table row separator colour = 10% gray
\usepackage[hidelinks]{hyperref}
\expandafter\def\expandafter\UrlBreaks\expandafter{\UrlBreaks%  save the current one
  \do\a\do\b\do\c\do\d\do\e\do\f\do\g\do\h\do\i\do\j%
  \do\k\do\l\do\m\do\n\do\o\do\p\do\q\do\r\do\s\do\t%
  \do\u\do\v\do\w\do\x\do\y\do\z\do\A\do\B\do\C\do\D%
  \do\E\do\F\do\G\do\H\do\I\do\J\do\K\do\L\do\M\do\N%
  \do\O\do\P\do\Q\do\R\do\S\do\T\do\U\do\V\do\W\do\X%
  \do\Y\do\Z}

\newcommand{\codim}{\mathrm{codim} }

\newcommand{\Spec}{\mathrm{Spec}}
\newcommand{\Proj}{\mathrm{Proj}}

\newcommand{\N}{ \mathbb{N}}

\newcommand{\ZZ}{ \mathbb{Z}}
\newcommand{\mapdefdot}[4]
{
                \left.
                \begin{array}{ll}
                        #1 & \dashrightarrow #2 \\
                        #3 & \mapsto #4
                \end{array}
        \right.
}

\newcommand{\U}{ \mathbb{U}}

\newcommand{\oo}{\mathcal{O}}
\newcommand{\pp}{\mathbb{P}}

\usepackage{amsmath,amsthm}
\usepackage{amsfonts}

\newtheorem{theorem}{Theorem}[section]
\newtheorem{propn}[theorem]{Proposition}
\newtheorem{corr}[theorem]{Corollary}

\newtheorem{defn}[theorem]{Definition}
\newtheorem{example}[theorem]{Example}
\newtheorem{remark}[theorem]{Remark}
\newtheorem{algorithm}{Algorithm}

\begin{document}

\maketitle
\begin{abstract} \noindent
Let $V$ be a closed subscheme of a projective space $\pp^n$. We give an algorithm to compute the Chern-Schwartz-MacPherson class, and the Euler characteristic of $V$ and an algorithm to compute the Segre class of $ V$. The algorithms can be implemented using either symbolic or numerical methods. The algorithms are based on a new method for calculating the projective degrees of a rational map defined by a homogeneous ideal. Relationships between the algorithms developed here and other existing algorithms are discussed. The algorithms are tested on several examples and are found to perform favourably compared to current algorithms for computing Chern-Schwartz-MacPherson classes, Segre classes and Euler characteristics.
\end{abstract}
\section{Introduction}
The topological Euler characteristic is an important invariant in a wide variety of areas of mathematics and has been studied by numerous authors in many different contexts. In this note we describe an algorithm to compute the Segre class, and an algorithm to compute the Chern-Schwartz-MacPherson class and the Euler characteristic of a projective variety over an algebraically closed field of characteristic zero. In particular, given the ideal $I$ defining a projective variety  $V$ in $\pp^n$  we will compute the pushforward to $\pp^n$ of both the Segre class of $V$ in $\pp^n$ and the Chern-Schwartz-MacPherson class of $V$ (we abuse notation and denote the pushforwards to $\pp^n$ as $s(V,\pp^n)$ and $c_{SM}(V)$ respectively). From $c_{SM}(V)$ we may immediately obtain the Euler characteristic of $V$, $\chi(V)$ using the well-known relation which states that $\chi(V)$ is equal to the degree of the zero dimensional component of $c_{SM}(V)$. The algorithm described may be implemented either symbolically, with the computations relying on Gr\"{o}bner bases calculations, or numerically using homotopy continuation. 

The methods to compute Segre classes and Chern-Schwartz-MacPherson classes described here are based on several known formulas due to Aluffi \cite{aluffi1999chern,aluffi2003computing}, and on the notion of the projective degrees of a rational map as expressed in Harris \cite{harris1992algebraic}. The main result of this note is Theorem \ref{projective_deg_theorem} which gives a method to compute projective degrees. 

We now give an example of the computation of the Segre class, the $c_{SM}$ class and the Euler characteristic for a singular projective variety. Note that since the variety $V$ considered in the example is singular the results $c_{SM}(V)$ and $\chi(V)$ could not be obtained with standard Chern class computations. 
\begin{example}
Let $V=V(I)$ be the subvariety of $\pp^4$ defined by the ideal $I=(4x_3x_2x_4x_1-x_0^3x_1,x_0x_1x_3x_4-x_2^3x_3)$ in $k[x_0,x_1,x_2,x_3,x_4]$. Also let $A_*(\pp^4) \cong \ZZ[h]/(h^{5})$ be the Chow ring of $\pp^4$. 

Using Algorithm \ref{algorithm:SegreAlg} with input $I$ we obtain the Segre class$$
s(V,\pp^n)=768h^4  - 128h^3  + 16h^2\in A_*(\pp^4).
$$Using Algorithm \ref{algorithm:CSMAlg} with input $I$ we obtain the Chern-Schwartz-MacPherson class $$
c_{SM}(V)=5h^4  + 8h^3  + 12h^2 \in A_*(\pp^4)
$$ and/or the Euler characteristic $
\chi(V)= 5.
$  \label{example:csm_ex}
\end{example}

The existence of a functorial theory of Chern classes for singular varieties, in terms of a natural transformation from
the functor of constructible functions to some nice homology theory, and its relation to the Euler characteristic, was conjectured by Deligne and Grothendieck in the 1960's. In the 1974 article~\cite{macpherson1974chern}, MacPherson proved the existence of such a transformation, introducing a new notion of Chern classes for singular algebraic varieties. Independently in the 1960's Schwartz \cite{schwartz1965classes} defined a theory of Chern classes for singular varieties in relative cohomology. It was later shown in a paper of Brasselet and Schwartz \cite{brasselet1981classes} that these two different notions were in fact equivalent. 

The desire to compute $c_{SM}$ classes explicitly is motivated partially by the relation to the Euler characteristic. In addition to being an important topological invariant, the Euler characteristic has applications to problems of maximum likliehood estimation in algebraic statistics \cite{huh2012maximum} and applications to string theory in physics such as \cite{collinucci2009d} and \cite{aluffi2010new}. The Chern-Schwartz-MacPherson class has also been directly related to problems in string theory in \cite{aluffi2009cherntadpole}. 

The organization of the remainder of this note is as follows. In Section \ref{BackgroundSection} we state the problem we wish to consider and review the definitions of the $c_{SM}$ class and Segre class. %We also give several important results relating to the computation of $c_{SM}$ classes. 

In Section \ref{section:Review}, we briefly review relevant background on the projective degrees of a rational map and state known formulas which expresses the Segre class and $c_{SM}$ class in terms of these projective degrees. Also in Section \ref{section:Review} we review previous algorithms for the computation of Segre and $c_{SM}$ classes. Specifically we review algorithms of Aluffi \cite{aluffi2003computing} and Eklund, Jost and Peterson \cite{Jost} for the computation of Segre classes and we review algorithms of Aluffi \cite{aluffi2003computing} and Jost \cite{jost2013algorithm} for the computation of $c_{SM}$ classes. Additionally we explain the relationship between the residual degrees computed by Eklund, Jost and Peterson in \cite{Jost} and the projective degrees in \eqref{eq:degR_eq_gi}. In light of this relationship one could see Algorithm \ref{algorithm:SegreAlg} and Algorithm \ref{algorithm:CSMAlg} as refinements of the algorithms of \cite{Jost} and \cite{jost2013algorithm} respectively; however we note that these methods are developed using very different theoretical tools, and a priori there is no obvious relationship between them.

In Section \ref{subsectionComp_CSM}, we describe the main ideas underlying the algorithms. Especially we state and prove Theorem \ref{projective_deg_theorem} which is the main result of this note and which gives a new formula for calculating the projective degrees of a rational map defined by a homogeneous ideal. 

In Section~\ref{algorithms}, we give a detailed description of the algorithms. In Algorithm~\ref{algorithm:csm_polar} we show how the result of Theorem~\ref{projective_deg_theorem} can be used to compute the projective degrees of a rational map using a computer algebra system. We then apply Algorithm~\ref{algorithm:csm_polar} to give Algorithm \ref{algorithm:SegreAlg} which computes the Segre class and Algorithm \ref{algorithm:CSMAlg} which computes the $c_{SM}$ class.

In Section \ref{section:Performance}, we discuss the performance of our algorithm to compute Segre classes (Algorithm \ref{algorithm:SegreAlg}) and our algorithm to compute the $c_{SM}$ class (Algorithm \ref{algorithm:CSMAlg}). Run time performance for Algorithm \ref{algorithm:SegreAlg} is compared with previous algorithms of Aluffi \cite{aluffi2003computing} and of Eklund, Jost and Peterson \cite{Jost} which also compute Segre classes. The results of the running time comparisons  for Segre classes are summarized in Table \ref{table:SegreResults}; we see that our algorithm performs favourably in most cases. Run time performance for Algorithm \ref{algorithm:CSMAlg} is compared with previous algorithms of Aluffi \cite{aluffi2003computing} and of Jost \cite{jost2013algorithm} which also compute the $c_{SM}$ class and/or the Euler characteristic. We also compare the Macaulay2 \cite{M2} implementation of Algorithm \ref{algorithm:CSMAlg} to the Macaulay2 built in routine euler which calculates Hodge numbers to compute the Euler characteristic. In all cases Algorithm \ref{algorithm:CSMAlg} performs favourably in comparison to other known algorithms. The results are summarized in Table \ref{table:results}.

The Macaulay2 \cite{M2} and Sage \cite{sage} implementations of our algorithm for computing $c_{SM}$ classes, Euler characteristics and Segre classes of projective varieties can be found at \url{https://github.com/Martin-Helmer/char-class-calc}. The Macaulay2 \cite{M2} implementation is also available as part of the ``CharacteristicClasses'' package in Macaulay2 version 1.7 and above and can be accessed using the option ``Algorithm$=>$ProjectiveDegree", see the Macually2 documentation \url{http://www.math.uiuc.edu/Macaulay2/doc/Macaulay2-1.7/share/doc/Macaulay2/CharacteristicClasses/html/} for further details.

\section{Problem and Definitions} \label{BackgroundSection}
Suppose $V$ is an arbitrary subscheme of a projective space $\pp^n$ over an algebraically closed field of characteristic zero. The problem we wish to consider is that of devising an effective and practical algorithmic method to compute the Segre class $s(V,\pp^n)$ and the Chern-Schwartz-MacPherson class $c_{SM}(V)$ as elements of the Chow ring of $\pp^n$. An algorithm which computes $c_{SM}(V)$ automatically give us the Euler characteristic $\chi(V)$, since this information is contained directly in $c_{SM}(V)$. 

%In this section we review the definition of the Segre class and the Chern-Swartz-MacPherson class and discuss some important results relating to the computation of the $c_{SM}$ class.

For $V$ a proper closed subscheme of a variety $W$, we may define the Segre class of $V$ in $W$ as $$
s(V,W)= \sum_{j\geq 1}(-1)^{j-1}\eta_*(\tilde{V}^j)
$$ where $\tilde{V}$ is the exceptional divisor of the blow-up of $W$ along $V$, $Bl_VW$, $\eta: \tilde{V} \to V$ is the projection and the class $\tilde{V}^k$ is the $k$-th self intersection of $\tilde{V}$. See Fulton \cite[\S 4.2.2]{fulton} for further details. In all cases considered in this note we will have $W=\pp^n$. Throughout the remainder of this section we consider possibly singular closed subschemes, $V$, of the projective space $\pp^n$ over an algebraically closed field~${k}$. 
 
The Chow ring of a $i$-dimensional nonsingular variety $V$ is   denoted by $A_*(V)= \oplus_{\ell=0}^iA_{\ell}(V)$, where $A_{\ell}(V)$ is the Chow group of $X$ having dimension $\ell$. Recall that the Chow groups of $V$ are the groups, $A_{\ell}$, of $\ell$-cycles on $V$ modulo rational equivalence. In what follows we will work only in the Chow ring of $\pp^n$, $A_*(\pp^n) \cong \ZZ[h]/(h^{n+1})$, where $h=c_1(\oo_{\pp^n}(1))$ is the equivalence class of a hyperplane in $\pp^n$, hence a hypersurface $W$ of degree $d$ in $\pp^n$ is represented as $[W]=d\cdot h$ in $A_*(\pp^n) $ (for more details see Fulton \cite{fulton}).

The total Chern class of a $j$-dimensional nonsingular projective variety $V$ is defined as $c(V)=c(TV) \cap [V]$ in the Chow ring of $V$, $A_*(V)$. As with $c_{SM}$ and Segre classes, we will abuse notation and write $c(V)$ for the pushforward to $\pp^n$ of the total Chern class of $V$. As a consequence of the Hirzebruch-Riemann-Roch theorem, we have that the degree of the zero dimensional component of the total Chern class of a projective variety is equal to the Euler characteristic, that is \begin{equation}
\int c(TV) \cap [V]=\chi(V). \label{eq:chern_euler_non_singular}
\end{equation}Here $\int \alpha$ denotes the degree of the zero dimensional component of the class $\alpha \in A_*(\pp^n)$, i.e. the degree of the part of $\alpha$ in $A_0(\pp^n)$.

There are several known generalizations of the total Chern class to singular varieties. All of these notions agree with $c(TV) \cap [V]$ for nonsingular $V$, however the Chern-Swartz-Macpherson class is the only one of these that satisfies a property analogous to (\ref{eq:chern_euler_non_singular}) for any $V$, i.e. \begin{equation}
\int c_{SM}(V)=\chi(V). \label{eq:csm_euler}
\end{equation} %As we will see below this relation follows from the functorial nature of the $c_{SM}$ class. 

 We review here the construction of the $c_{SM}$ classes, given in the manner considered by MacPherson \cite{macpherson1974chern}. For a scheme $V$, let $\mathcal{C}(V )$ denote the abelian group of finite linear combinations $\sum_W m_W \mathfrak{1}_W$, where $W$ are (closed) subvarieties of $V$, $m_W \in \ZZ$, and $\mathfrak{1}_W$ denotes the function that is $1$ in $W$, and $0$ outside of $W$. Elements  $f\in \mathcal{C}(V )$ are known as constructible functions and the group  $\mathcal{C}(V )$ is referred to as the group of constructible functions on $V$. To make $\mathcal{C}$ into a functor we let $\mathcal{C}$ map a scheme $V$ to the group of constructible functions on $V$ and a proper morphism $f: V_1 \to V_2$  is mapped by $\mathcal{C}$ to $$\mathcal{C}(f)(\mathfrak{1}_W)(p)=\chi(f^{-1}(p) \cap W), \;\;\; W \subset V_1, \; p\in V_2 \; \mathrm{a \; closed \; point}.$$ 

Another functor from algebraic varieties to albelian groups is the Chow group functor $\mathcal{A}_*$. The $c_{SM}$ class may be realized as a natural transform between these two functors.   
\begin{defn}
The Chern-Schwartz-MacPherson class is the unique natural transformation between the constructible function functor and the Chow group functor, that is $c_{SM}: \mathcal{C}\to \mathcal{A}_*$ is the unique natural transform satisfying: \begin{itemize}
\item (\textit{Normalization}) $ c_{SM}(\mathfrak{1}_V)=c(TV) \cap [V] $ for $V $ non-singular  and complete.
\item (\textit{Naturality}) $f_{*}(c_{SM}(\phi))=c_{SM}(\mathcal{C}(f)(\phi))$, for $f:X \to Y$ a proper transform of projective varieties, $\phi$ a constructible function on $X$. 
\end{itemize}
\end{defn}
For a scheme $V$ let $V_{red}$ denote the support of $V$, the notation $c_{SM}(V)$ is taken to mean $c_{SM}(\mathfrak{1}_V)$ and hence, since $\mathfrak{1}_V=\mathfrak{1}_{V_{red}} $, we denote $c_{SM}(V)=c_{SM}(V_{red})$. 

To see how the $c_{SM}$ class satisfies the relation (\ref{eq:csm_euler}) consider the morphism $\mathfrak{f}:V \to \mathrm{point} ,$ applying the naturality property of the $c_{SM}$ class we have \small \begin{equation*}
\mathfrak{f}_{*}(c_{SM}(V))= c_{SM}(\mathcal{C}(\mathfrak{f})(\mathfrak{1}_V ))=c_{SM}(\chi(V)\mathfrak{1}_{\mathrm{point}})=\chi(V)c_{SM}(\mathrm{point})=\chi(V)[\mathrm{point}].
\end{equation*} \normalsize This gives us (\ref{eq:csm_euler}). 
Note that the $c_{SM}$ classes (and constructible functions) also satisfy the same inclusion/exclusion relation as the Euler characteristic. Specifically for $V_1,V_2$ subschemes of $\pp^n$ we have
, i.e. for the Euler characteristic we have $$
\chi(V_1 \cup V_2) = \chi (V_1) \chi(V_2) -\chi(V_1\cap V_2).
$$ Constructible functions inherit this property from the Euler characteristic via the definition of the constructible function functor, specifically we have $\mathfrak{1}_{V_1 \cup V_2}= \mathfrak{1}_{V_1 }+\mathfrak{1}_{V_2}-\mathfrak{1}_{V_1 \cap V_2}$.  From this we see that the $c_{SM} $ classes will also possess an inclusion/exclusion property, giving us the relation  
\begin{equation}
c_{SM}(V_1 \cap V_2)=c_{SM}(V_1)+c_{SM}(V_2)-c_{SM}(V_1 \cup V_2).
\label{eq:csm_inclusion_exclusion}
\end{equation}

To give the reader a more intuitive understanding of the geometric information contained in the $c_{SM}$ class we recall a result of Aluffi \cite{aluffi2013euler} which states that when $V $ is a subscheme of $\pp^n$ then $c_{SM}(V)$ contains the Euler characteristics of $V$ and those of general linear sections of $V$ for each codimension. In this way one may consider $c_{SM}(V)$ as a more refined version of the Euler characteristic in $\pp^n$. Specifically, if $\dim(V)=m$, starting from $c_{SM}(V)$ we may directly obtain the list of invariants $$\chi(V),\chi(V \cap L_1), \chi(V \cap L_1\cap L_2), \dots,\chi(V \cap L_1\cap \cdots \cap  L_m) $$ where $L_1,\dots, L_m$ are general hyperplanes. Conversely from the list of Euler characteristics above we could obtain $c_{SM}(V)$, i.e. there exists an involution between the Euler characteristics of general linear sections and the $c_{SM}$ class in this setting. This relationship is given explicitly in Theorem 1.1 of Aluffi \cite{aluffi2013euler}, we give an example of this below. 

\begin{example} 
Consider again the subvariety of $\pp^4$ given by $V=V(4x_3x_2x_4x_1-x_0^3x_1,x_0x_1x_3x_4-x_2^3x_3)$. We know from Example \ref{example:csm_ex} that $c_{SM}(V)=5h^4  + 8h^3  + 12h^2 $. To obtain the Euler characteristics of the general linear sections of $V$ we may apply an involution formula given by Aluffi in \cite[Theorem 1.1]{aluffi2013euler}, specifically:
\begin{itemize}
\item First consider the polynomial $p(t)= 5+8t+12t^2 \in \ZZ[t]/(t^5) $ given by the coefficients of the $c_{SM}$ class above.
\item Next apply Aluffi's involution $$p(t) \mapsto \mathcal{I}(p):= \frac{t \cdot p(-t-1)+p(0)}{t+1} =12t^2+4t+5.$$ 
\end{itemize} This gives $\chi(V)=5,$ $\chi(V\cap L_1)=(-1)^1\cdot 4=-4,$ and $\chi(V\cap L_1 \cap L_2)=(-1)^2\cdot 12=12$ where $L_1$ and $L_2$ are general hyperplanes in $\pp^4$. \label{ex:AluffiInvolution}
\end{example}

\section{Background and Review} \label{section:Review}
As in the previous section we consider possibly singular closed subschemes, $V$, of the projective space $\pp^n$ over $k$, an algebraically closed field of characteristic zero. 

We review the definition of the projective degrees of a rational map in \S\ref{subsection:projdeg}. In \S\ref{subsection:segre} we first review a result of Aluffi \cite{aluffi2003computing} which gives an explicit expression for the Segre class $s(V,\pp^n)$ in terms of the projective degrees in Proposition \ref{propn:aluffi_segre}. We then discuss previous algorithms to compute the Segre class $s(V,\pp^n)$. 

In \S\ref{subsection:csm} we review a result of Aluffi \cite{aluffi1999chern} which allows one to compute the Chern-Schwartz-MacPherson class of a hypersurface by computing certain Segre classes, stated in Proposition \ref{propn:allufi_csm_hyper}. In Proposition \ref{propn:csm_higher_codim} we state a general version of the inclusion/exclusion property of $c_{SM}$ classes which will allow for the computation of the $c_{SM}$ class in codimension greater than one. We also discuss previous algorithms which use the result stated in Proposition \ref{propn:allufi_csm_hyper} to calculate $c_{SM}$ classes. Finally we give a result of Aluffi \cite{aluffi2003computing} which gives an expression for the $c_{SM}$ class of a hypersurface in terms of the projective degrees of a certain rational map in Theorem \ref{Csm_computing _thorem}.  
\subsection{Projective Degrees} \label{subsection:projdeg}
  Here we recall the definition of the projective degrees of a rational map as in Harris \cite{harris1992algebraic}; the computation of these projective degrees will allow for the calculation of Segre and $c_{SM}$ classes using Algorithms \ref{algorithm:csm_polar}, \ref{algorithm:SegreAlg} and \ref{algorithm:CSMAlg}. 

Consider a rational map $\phi : \pp^n \dashrightarrow \pp^m$. In the manner of  Harris (Example 19.4 of \cite{harris1992algebraic}) we may define the \textit{projective degrees} of the map $\phi$ as a list of integers $(g_0, \dots , g_n)$ where \begin{equation}
g_i=\mathrm{card} \left(\phi^{-1} \left( \pp^{ m-i} \right) \cap \pp^{i} \right). \label{eq:harris_proj_deg_def}
\end{equation} Here $\pp^{ m-i} \subset  \pp^m$ and $\pp^{i} \subset \pp^n$ are general hyperplanes of dimension $m-i$ and $i$ respectively and $\mathrm{card}$ is the cardinality of a zero dimensional set. Note that points in $\left(\phi^{-1} \left( \pp^{ m-i} \right) \cap \pp^{i} \right)$ will have multiplicity one (this follows from the Bertini theorem of Sommese and Wampler \cite[\S A.8.7]{sommese2005numerical}). Let $\Gamma_{\phi} \subset \pp^n \times \pp^m$ be the graph of $\phi$. The numbers $g_i$ are also used by Aluffi \cite{aluffi1999chern,aluffi2003computing,aluffi2013euler}, where the class $[\Gamma_{\phi}]$ is pushed forward to a class $[G]\in A_*(\pp^n)$ by the projection map onto the first factor of $\pp^n\times \pp^m$. Aluffi refers to the class $[G]$ as the class of the shadow of the graph of the map $\phi$.  Specifically, take $t$ to be the pull-back of the hyperplane class from the $\pp^m$ factor of $\pp^n\times \pp^m$ and let $\pi: \Gamma_{\phi} \to \pp^n$ be the projection. In the notation of \cite{aluffi2003computing}, the \textit{shadow} of $\Gamma_{\phi}$ is the class \begin{equation}
[G]=g_0+g_{1}h+ \cdots + g_nh^n\in A_* (\pp^n),  \label{eq:GammaDef}
\end{equation} where $g_i=\deg(\pi_*(t^i \cdot [\Gamma_{\phi}]))$, these $(g_0, \dots , g_n)$ are also the projective degrees of the map $\phi$.

We give a method to compute the projective degrees $g_i$ in Theorem \ref{projective_deg_theorem} below. 

\subsection{Segre classes} \label{subsection:segre}
In this subsection we state a result of Aluffi \cite{aluffi1999chern} (Proposition \ref{propn:aluffi_segre}) which can be used to calculate Segre classes of projective varieties. When combined with result of Theorem \ref{projective_deg_theorem} this yields our algorithm to compute Segre classes of projective varieties described in Algorithm \ref{algorithm:SegreAlg}. We also review several previous results on the computation of Segre classes, the first due to Aluffi \cite{aluffi1999chern} and the second due to Eklund, Jost and Peterson \cite{Jost}. 

In (\ref{eq:degR_eq_gi}) we make explicit the relationship between the projective degrees of a rational map and the degrees of the residual set considered in \cite{Jost}. 

Aluffi \cite{aluffi2003computing} gives the following result which allows for the computation of the Segre class of $V$ in $\pp^n$ for $V$ a subscheme of $\pp^n$.
\begin{propn}[Proposition 3.1 \cite{aluffi2003computing}] Let $I=(f_0, \dots , f_m) \subset k[x_0, \dots , x_n] $ be a homogeneous ideal defining a scheme $V\subset \pp^n$ and let $h=c_1\left( \oo_{\pp^n}(1) \right)$ be the class of a hyperplane in $A_*(\pp^n)$. Since $I$ is homogeneous we may assume that the $\deg (f_i) =d$ for all $i$. Let $\phi: \pp^n \dashrightarrow \pp^m $ be the rational map specified by  $$
p \mapsto (f_0(p):\cdots: f_m(p)),$$ let $(g_0,\dots,g_n)$ be the projective degrees of $\phi$ and let $\Gamma_{\phi} \subset \pp^n \times \pp^m$ be the graph of $\phi$. Write $[G]$ for the class of the shadow of the graph of the map $\phi$ (see \eqref{eq:GammaDef}), i.e. $$[G] =g_0+g_1h+ \cdots +g_{n-1}h^{n-1}+g_nh^n $$ in $A_*(\pp^n)\cong \ZZ[h]/(h^{n+1})$. Then we have: \begin{equation}
s(V,\pp^n) = 1 - c(\oo(dh))^{-1} \cap \left(\sum_{i=0}^n \frac{g_ih^i}{c(\oo(dh))^i)} \right) \; \in A_*(\pp^n). 
\end{equation} \label{propn:aluffi_segre} \end{propn} To use Proposition \ref{propn:aluffi_segre}, Aluffi \cite{aluffi2003computing} notes that $\Gamma_{\phi}$ can be obtained explicitly as $\Gamma_{\phi}$ is isomorphic to the blow-up of $\pp^n$ along $V$, and once $\Gamma_{\phi}$ is known the class $[G]$ can be computed directly.  Specifically the algorithm of Aluffi is as follows, \begin{itemize}
\item obtain $\Gamma_{\phi}$ explicitly (by computing $Bl_{V}\pp^n \cong \Gamma_{\phi}$, that is the blow-up of $\pp^n$ along $V$)
\item intersect $\Gamma_{\phi}$ with general hyperplanes 
\item project the intersections down to $\pp^n$, and compute the degree of the projections to obtain the class of the shadow of the graph, $[G]$.
\end{itemize} Hence the main computational cost for finding Segre classes using the method of \cite{aluffi2003computing} is that of finding the blow-up of $\pp^n$ along $V$. 

Another method for computing Segre classes was given by Eklund, Jost and Peterson \cite{Jost}. This method does not use the relation between the class of the shadow of the graph $[G]$ (see (\ref{eq:GammaDef})) and the Segre class $s(V,\pp^n)$; we summarize the result in Proposition \ref{propn:JostSegre} below. 
\begin{propn}[Theorem 3.2 \cite{Jost}]
Let $V \subset \pp^n$ be a subscheme of dimension $\varrho$ defined by a non-zero homogeneous ideal $I=(f_0,\dots, f_m) \subset k[x_0, \dots , x_n]$ with the generators $f_i$ having degree $d$.  Let $$s(V,\pp^n)=s_n+ \cdots + s_0 h^n \in A_*(\pp^n)\cong \ZZ[h]/(h^{n+1})$$ be the Segre class of $V$ in $\pp^n$. For $n-\varrho \leq j \leq n$ and general elements $\gamma_1, \dots , \gamma_j$ let $J=(\gamma_1, \dots , \gamma_j) $ and let $R_j \subset \pp^n$ be the subscheme defined by $J:I^{\infty  }$. Then we have $$
d^j= \deg(R_j)+\sum_{i=0}^{j-(n-\varrho)} {j \choose j-(n-\varrho)-i} d^{j-(n-\varrho)-i}s_i .
$$
\label{propn:JostSegre}

\end{propn} To apply Proposition \ref{propn:JostSegre} to compute $s(V,\pp^n)$, Eklund, Jost and Peterson \cite{Jost} use the following method.
\begin{itemize}
\item $V=V(I)$, say $d$ is the degree of the homogeneous generators of $I$.
\item Pick general degree $d$ polynomials $\omega_1, \dots , \omega_j$ in $I$.
\item For $j=n-\dim V=\codim(V)$ to $j=n$ do:
\begin{itemize}
\item Set $J=(\omega_1,\dots, \omega_j)$ and let $R_j$ be the scheme defined by $J:I^{\infty}$.
\item Compute $\deg(R_j)$.
\item Set $p=j-\codim(V)$, \begin{equation}
s_p=d^j-\deg(R_j)-\sum_{i=1}^{p-1} { j \choose p-i} d^{p-i}s_i. \label{eq:JostSegre}
\end{equation}

\end{itemize}
\end{itemize}
Hence the main computational cost in the algorithm of Eklund, Jost and Peterson \cite{Jost} is the computation of $\deg(R_j)$. When done symbolically, this means the main cost arises from the computation of the saturation $J:I^{\infty}$ for each $j$. Eklund, Jost and Peterson \cite{Jost} also explain that $\deg( R_j)$ can be computed numerically using homotopy continuation in Bertini \cite{Bertini}.

There is, in fact, an explicit relationship between the projective degrees $(g_0,\dots,g_n)$ of a rational map $\phi$ defined by an ideal $I$ (or equivalently the class $[G]$ of the shadow of the graph $\Gamma_{\phi}$ (\ref{eq:GammaDef})) and the degrees of the residual sets $R_j$ in Proposition \ref{propn:JostSegre}. Specifically let $V=V(I)$ be a subscheme of $\pp^n$ where $I=(f_0,\dots,f_m)$ is a homogenous ideal in $k[x_0,\dots,x_n]$ and let $[G]=g_0+g_1h+ \cdots +g_{n-1}h^{n-1}+g_nh^n  \in A_*(\pp^n )$ be the class of the shadow of the graph of $\phi$ (as in Proposition \ref{propn:aluffi_segre}). Since $I$ is homogenous we may assume that $\deg(f_j)=d$ for all $ i=1,\dots,m$. Take $\nu=\codim(Y)$. Let $$s(V,\pp^n)= s_n+\cdots+s_0h^n \in A_*(\pp^n)$$ be the Segre class of $V$ in $\pp^n$ and let $ \tilde{s}_0=1$, $ \tilde{s}_1= \cdots =  \tilde{s}_{\nu-1}=0$ and $\tilde{s_i}=-  s_{i-\nu} $ for $i\geq \nu$. Note that $s_n=\cdots = s_{\nu+1}=0$, i.e. $s_{\nu}$ is the first nonzero coefficent. In \cite{jost2013algorithm} Jost gives the following expression relating the $g_j$ in the class of the graph $\Gamma_I$ to the Segre class,  \begin{equation}
g_j=\sum_{i=0}^j {j \choose i} d^{j-i}\tilde{s_i}, \label{eq:gfirst}
\end{equation} which is obtained by rearranging and simplifying the expression of Proposition \ref{propn:aluffi_segre}. The result of Proposition \ref{propn:JostSegre} gives the following expression for $\deg(R_j)$ when $j=\nu, \dots, n$, \begin{equation}
 \deg(R_j)=d^j-\sum_{i=0}^{j-(n-\nu)} {j \choose j-(n-\nu)-i} d^{j-(n-\nu)-i}s_i. \label{eq:degR1}
\end{equation}Reindexing the summation in (\ref{eq:degR1}) we have $$
 \deg(R_j)=d^j-\sum_{i=\nu}^{j} {j \choose i} d^{j-i}s_{i-\nu}, \;\;\; \mathrm{for \; }j=\nu,\dots, n.
$$ Since $ \tilde{s}_0=1$ and $ \tilde{s}_1= \cdots =  \tilde{s}_{\nu-1}=0$ we may rewrite the expression (\ref{eq:gfirst}) for $g_j$ as $$
g_j=d^j -\sum_{i=\nu}^j {j \choose i} d^{j-i}{s_{i-\nu}}, \;\;\; \mathrm{for \; }j=\nu,\dots, n, 
$$ and $g_j=d^j$ for $j=0, \dots, \nu -1$. Hence we have that \begin{equation}
\deg(R_j)=g_j \; \mathrm{for} \; j=\nu,\dots, n.\label{eq:degR_eq_gi}
\end{equation}

In light of (\ref{eq:degR_eq_gi}) we observe that the method for computing Segre classes of Eklund, Jost and Peterson \cite{Jost} stated in Proposition \ref{propn:JostSegre} computes the same values as the result of Theorem \ref{projective_deg_theorem}, and in fact, the method of Theorem \ref{projective_deg_theorem} can be seen as a refinement of the method of \cite{Jost}. In both cases similar systems of equations are considered, however we will see below that the method of Algorithm \ref{algorithm:csm_polar} tends to perform better. 

\subsection{Chern-Schwartz-MacPherson Classes} \label{subsection:csm}
In this subsection we review previous results on the calculation of the $c_{SM}$ class of a projective variety due to Aluffi \cite{aluffi1999chern,aluffi2003computing} and Jost \cite{jost2013algorithm}. We then state Theorem \ref{Csm_computing _thorem}, a result of Aluffi \cite{aluffi2003computing}, which when combined with Corollary \ref{corr:polar_deg_csm} below allows for the computation of the Chern-Schwartz-MacPherson class of a projective variety in the manner described in Algorithm \ref{algorithm:CSMAlg}.     

A tangible realization of the $c_{SM}$ classes, in the case of hypersurfaces, was given by Aluffi in Theorem I.4 of \cite{aluffi1999chern}. We state the result in the following proposition. \begin{propn}[Theorem I.4 \cite{aluffi1999chern}] Let $V=V(f)$ be a hypersurface  of $\pp^n$, for some $f\in k[x_0,\dots, x_n]$, and asume without loss of gernality that $f$ is squarefree (since $c_{SM}(V)=c_{SM}(V_{red})$) then \begin{equation}
c_{SM}(V)=c(T\pp^n) \cap \left( s(V,\pp^n)+ \sum_{m=0}^n \sum_{j=0}^{n-m} {n-m \choose j}(-V)^j \cdot (-1)^{n-m-j}s_{m+j}(Y, \pp^n)\right) \label{eq:allufi_csm_hyper}
\end{equation} where $s(V,\pp^n)$ is the Segre class of $V$ in $\pp^n$, and $Y$ is the singularity subscheme of $V$. That is, $Y$ is the scheme defined by the vanishing of the partial derivatives of $f$.\label{propn:allufi_csm_hyper}
\end{propn}

Note that the inclusion/exclusion relation for $c_{SM}$ classes will allow us to reduce all computation of $c_{SM}$ classes to the case of hypersurfaces. In the case of subvarieties of $\pp^n$ we have the following proposition, discussed informally by Aluffi \cite{aluffi2003computing}; Proposition \ref{propn:csm_higher_codim} follows directly from (\ref{eq:csm_inclusion_exclusion}).

\begin{propn}
Let $   V=X_1 \cap \cdots \cap X_{r} = V(f_1) \cap \cdots \cap V(f_r)$ be a subscheme of $\pp^n = \Proj(k[x_0, \dots , x_n])$. Write the polynomials defining $V$ as $F=(f_1,\dots , f_r)$ and let $F_{ \left\lbrace S \right\rbrace } = \prod_{i \in S} f_i $ for $S \subset \left\lbrace 1, \dots , r\right\rbrace$ . Then, $$
c_{SM}(V)= \sum_{S \subset \left\lbrace 1, \dots , r\right\rbrace} (-1)^{|S|+1}c_{SM} \left(V( F_{\left\lbrace S \right\rbrace } )\right)
$$  where $|S|$ denotes the cardinality of the integer set $S$. \label{propn:csm_higher_codim}
\end{propn}

 In \cite{aluffi2003computing}, Aluffi uses Proposition \ref{propn:aluffi_segre} and Proposition \ref{propn:allufi_csm_hyper} to give an algorithm to compute the $c_{SM}$ class of hypersurface in $\pp^n$ (this algorithm can be extended to higher codimension using Proposition \ref{propn:csm_higher_codim}). That is for a hypersurface $V=V(f) $ in $\pp^n$ and  $Y$ the singularity scheme of $V$ (that is the scheme defined by the zeros of the partial derivatives of $f$) the algorithm of Aluffi \cite{aluffi2003computing} computes $s(Y,\pp^n)$ by finding the blow up, as described above (immediately following Proposition \ref{propn:aluffi_segre}), and then applying Proposition \ref{propn:allufi_csm_hyper}. Thus the main computational step of the algorithm is to compute the blow-up of $\pp^n$ along $Y$ for each hypersurface. This can be implemented using any algorithm which computes the Rees algebra of $Y$.

An alternative method for computing $c_{SM}$ classes was given by Jost in \cite{jost2013algorithm}. This method also uses (\ref{eq:allufi_csm_hyper}) to give an expression for the $c_{SM}$ class of a hypersurface, however Jost computes the class $s(Y,\pp^n)$ by applying the method of \cite{Jost} stated in Proposition \ref{propn:JostSegre} to compute Segre classes by calculating the degrees of residual sets.

Let $V$ be a hypersurface of $\pp^n$ defined by the homogeneous polynomial ideal $(f) $ in $k[x_0, \dots , x_n]$, and since $c_{SM}(V)=c_{SM}(V_{red})$ we assume that $f \in k[x_0, \dots , x_n]$ is squarefree. Using the partial derivatives of $f$ we define a rational map $ \varphi : \pp^n  \dashrightarrow \pp^n$,  \begin{equation}
\varphi:p \mapsto \left( \frac{\partial f}{\partial x_0}(p): \cdots : \frac{\partial f}{\partial x_n}(p) \right). \label{eq:polar_map_def}
\end{equation}
This map is referred to as the \textit{polar map} or \textit{gradient map} \cite{dolgachev2000polar}. 

\begin{theorem}[Aluffi \cite{aluffi2003computing} Theorem 2.1] Assume, without loss of generality, that $f\in k[x_0,\dots,x_n]$ is squarefree. Let $V=V(f)$ and let $(g_0,\dots,g_n)$ be the projective degrees of the polar map $\varphi$ (\ref{eq:polar_map_def}), we have the following equality in $A_*(\pp^n)=\ZZ[h]/h^{n+1}$
\begin{equation}
c_{SM}(V)=(1+h)^{n+1}-\sum_{j=0}^{n}g_j (-h)^j(1+h)^{n-j}.
\end{equation} \label{Csm_computing _thorem}
\end{theorem}
Note that Theorem \ref{Csm_computing _thorem} follows from substituting the result of Proposition \ref{propn:aluffi_segre} (as stated in (\ref{eq:Segre_gs})) into the result of Proposition \ref{propn:allufi_csm_hyper}, (\ref{eq:allufi_csm_hyper}).

 \begin{remark}
The following special case is from Suwa \cite{suwa1997classes}. Let $X$ be a smooth subvariety of $\pp^n$ which is a global complete intersection, further suppose that $X=V(f_0,\dots,f_r)$ with $d_i=\deg f_i$, then we have \begin{equation}
c_{SM}(X)=c(X)=(1+h)^{n+1} \cdot \prod_{i=0}^{\codim X} \frac{d_ih}{1+d_ih} \;\;\; \mathrm{in \;} A_*(\pp^n),
\end{equation} recall that $c(X)=c(TX)\cap [X]$ is the total Chern class of the smooth variety $X$. \label{remark:smooth_csm}
 \end{remark}
We note that using Remark \ref{remark:smooth_csm} the computation of $c_{SM}$ classes could be made much more efficient in the particular case where the input scheme is a complete intersection which is known to be smooth.

\section{ Main Theoretical Results} \label{subsectionComp_CSM}
In this section we describe the main ideas underlying the algorithms given in the next section.  We begin by stating and proving the main result of this note, Theorem \ref{projective_deg_theorem}. This theorem gives a method to compute the projective degrees of a rational map defined by a homogeneous ideal. %This theorem will be used to construct an algorithm to compute the projective degrees using a computer algebra system, presented in Algorithm \ref{algorithm:csm_polar}. Algorithm \ref{algorithm:csm_polar} will then be to construct Algorithm \ref{algorithm:SegreAlg} which computes the Segre class $s(V,\pp^n)$ of a subscheme $V$ of $\pp^n$ and Algorithm \ref{algorithm:CSMAlg} which computes $c_{SM}(V)$ and/or $\chi(V)$.

%The result of Theorem \ref{projective_deg_theorem} is used to construct an algorithm to compute the projective degrees using a computer algebra system, presented in Algorithm \ref{algorithm:csm_polar}. Algorithm \ref{algorithm:csm_polar} is in turn used to construct Algorithm \ref{algorithm:SegreAlg} which computes the Segre class $s(V,\pp^n)$ of a subscheme $V$ of $\pp^n$ and Algorithm \ref{algorithm:CSMAlg} which computes $c_{SM}(V)$ and/or $\chi(V)$.

\begin{theorem}
Let $I=(f_0,\dots,f_m)$ be a homogeneous ideal in $k[x_0, \dots , x_n]$ defining a $\varrho$-dimensional scheme $V=V(I)$, and assume, without loss of generality that all the polynomials $f_i$ generating $I$ have the same degree.  The projective degrees $(g_0,\dots , g_n)$ of $ \phi : \pp^n  \dashrightarrow \pp^m$,  $$
\phi:p \mapsto \left( f_0(p): \cdots : f_m(p) \right),
$$ are given by \begin{equation}
g_i= \dim_k \left( k[x_0, \dots , x_n,T]/(P_1 +\cdots +P_{i}+L_1+\cdots +L_{n-i}+L_A+S)\right).
\end{equation} Here $P_{\ell},L_{\ell},L_A$ and $S$ are ideals in $k[x_0,\dots , x_n,T]$  with \begin{align*}
P_{\ell}& = \left( \sum_{j=0}^m \lambda_{\ell,j} f_j\right), \;\;\; \lambda_{\ell,j} \mathrm{\; a \; general \; scalar \; in \;}k,\; \ell = 1,\dots,n,\\
S& =  \left( 1-T \cdot \sum_{j=0}^m \vartheta_{j} f_j \right), \;\;\; \vartheta_{j} \mathrm{\; a \; general \; scalar \; in \;}k,\\
L_{\ell}&= \left( \sum_{j=0}^n \mu_{\ell,j} x_j \right), \;\;\; \mu_{\ell,j} \mathrm{\; a \; general \; scalar \; in \;}k,\; \ell = 1,\dots,n,\\
L_{A}&= \left(1- \sum_{j=0}^n \nu_{j} x_j \right), \;\;\; \nu_{j} \mathrm{\; a \; general \; scalar \; in \;}k.\\
\end{align*}\label{projective_deg_theorem} 
Additionally $g_0=1$.  
\end{theorem}
\begin{proof}
First we observe that by (\ref{eq:gfirst}) we have that $g_0=1$. Fix some $i=1,\dots,n$. For the rational map $\phi$ the projective degrees (see (\ref{eq:harris_proj_deg_def})) are given by  
$$g_i=\mathrm{card} \left(\phi^{-1} \left( \pp^{ m-i} \right) \cap \pp^{i} \right).$$ The inverse image under $\phi $ of a general hyperplane $\pp^{m-1}$ in $\pp^m$  is $$
\phi^{-1} \left( \pp^{m-1} \right)=V \left(\sum_{j=0}^m \lambda_j f_j \right)-V(f_0,\dots ,f_m) \subset \pp^n, \;\;\; \mathrm{for}\; \lambda_j \mathrm{\; a \; general \; scalar \; in \;}k
$$ and letting $$L_{\ell}=\left( \sum_{j=0}^n \mu_{\ell,j} x_j \right), \;\;\; \mu_{\ell,j} \mathrm{\; a \; general \; scalar \; in \;}k$$ for each ${\ell}$, this gives  $$
g_i=\mathrm{card} \left( \bigcap_{\ell=1}^{i}V \left(\sum_{j=0}^{m} \lambda_{\ell,j} f_j\right) \cap \bigcap_{{\ell}=1}^{n-i}V(L_{\ell}) -V(f_0,\dots ,f_m)\right). 
$$ Now let $$
W=\bigcap_{\ell=1}^{i}V \left(\sum_{j=0}^{m} \lambda_{\ell,j} f_j\right) \cap \bigcap_{{\ell}=1}^{n-i} V\left(  L_{\ell} \right) ,
$$ so $g_i=\mathrm{card}\left(W-V(f_0,\dots, f_m) \right)$. Let $\widetilde{W}=W-V(f_0,\dots ,f_m)$. By the Bertini theorem of Sommese and Wampler \cite[\S A.8.7]{sommese2005numerical} there exists open dense subsets $U_1 \subset \pp^{i\times m}$ and $U_2\subset \pp^{n-i \times n}$ such that for $ \lambda \in U_1 $ and $\mu \in U_2$, $ \widetilde{W}$ has dimension $0$ and $\oo_{\widetilde{W},p}$ is a regular local ring (equivalently the Jacobian matrix of the generators of $W$ evaluated at points in $  W -V(f_0,\dots ,f_m)$ has rank $n$). In what follows we take $ \lambda \in U_1 $ and $\mu \in U_2$.  Let us write $W -V(f_0,\dots ,f_m)=\left\lbrace p_0,\dots, p_s\right\rbrace$. Then $$U_3=\pp^m- \bigcup_{i=0}^sV \left(f_0(p_i)x_0+\cdots +f_m(p_i)x_m \right)$$ is open and dense in $\pp^m$, because $(f_0(p_i),\dots,f_m(p_i))\neq (0,\dots,0)$ for all $i$. Take $ \vartheta=(\vartheta_0,\dots ,\vartheta_m)\in U_3$; then $$ W\cap V \left( \sum_{j=0}^m \vartheta_{j} f_j \right)  -V(f_0,\dots ,f_m)$$ is empty.  Now consider the ideals $L_{\ell}$ and $ \left(\sum_{j=0}^m \lambda_{\ell,j} f_j \right)$ as ideals in the ring $ k[x_0, \dots , x_n,T]$, and define $V_S=V(S)$ where $$ S=\left( 1-T \cdot \sum_{j=0}^m \vartheta_{j} f_j\right) $$ is an ideal in $k[x_0, \dots , x_n,T]$.
%and let $$
%W_i= \bigcap_{\ell=1}^{i}V \left(\sum_{j=0}^m \lambda_{\ell,j} f_j \right).
%$$ Take $$
%B= \bigcup_{\ell=1}^{i}V \left(\sum_{j=0}^m \lambda_{\ell,j} f_j \right), 
 %$$and note that $\pp^n -B$ is a Zariski open, and dense, subset of $\pp^n$. This implies that there is an open and dense set of points $ \vartheta=(\vartheta_0,\dots, \vartheta_m)$ such that $$B-V\left( \sum_{j=0}^m \vartheta_{j} f_j\right)$$ is open and dense in $B$. Define $$V_S=V(S) \; \mathrm{where \;} S=\left( 1-T \cdot \sum_{j=0}^m \vartheta_{j} f_j\right) \mathrm{\; is \; an \;ideal \; in \;}k[x_0, \dots , x_n,T].$$  
For a point $p\in V(f_0,\dots,f_m)$ we have that $$f_j(p)=0, \; \;\; j=0,1,\dots, m$$ which implies that $p$ is not in $V_S$ since $p$ cannot be a solution to the equation $ 1-T \cdot \sum_{j=0}^m \vartheta_{j} f_j=0$. Now take $p \in W-V(f_0,\dots ,f_m)$ then $$T_p=\frac{1}{\sum_{j=0}^m \vartheta_{j} f_j(p)}$$ is well defined since for $\vartheta\in U_3$ we have that $W\cap V \left( \sum_{j=0}^m \vartheta_{j} f_j \right)  -V(f_0,\dots ,f_m) $ is empty, so $(p,T_p) \in V_S$. Now let $\widehat{W} \subset \pp^n \times \mathbb{A}^1$ be the variety given by a linear embedding of $W$ in $ \pp^n \times \mathbb{A}^1$, where $\mathbb{A}^1=\Spec(k[T])$. We have \begin{equation}
\pi(\widehat{W}\cap V_S)=W-V(f_0,\dots,f_m),
\end{equation} where $\pi $ is the projection $\pi: \pp^n \times \mathbb{A}^1 \mapsto \pp^n$, and in particular $$
\mathrm{card}(\widehat{W}\cap V_S)=\mathrm{card}(W-V(f_0,\dots,f_m)).
$$

Rather than considering the intersection $\widehat{W}\cap V_S $ in $  \pp^n \times \mathbb{A}^1$ we take  $W\subset \mathbb{A}^{n}$ i.e. we dehomogenize by taking $$W= \bigcap_{\ell=0}^{i}V \left(\sum_{j=0}^{m} \lambda_{\ell,j} f_j\right) \cap \bigcap_{{\ell}=1}^{n-i}V(L_{\ell}) \cap V(L_A)\subset \mathbb{A}^n$$ and consider the intersection $\widehat{W}\cap V_S$ in $\mathbb{A}^{n+1}$. As the points in $\phi^{-1} \left( \pp^{ m-i} \right) \cap \pp^{i} $ have multiplicity one (by the Bertini theorem of Sommese and Wampler~\cite[\S A.8.7]{sommese2005numerical}) the cardinality of the zero dimensional set $$ \bigcap_{\ell=0}^{i}V \left(\sum_{j=0}^{m} \lambda_{\ell,j} f_j\right) \cap \bigcap_{{\ell}=1}^{n-i}V(L_{\ell}) \cap V(L_A) \cap V_S \subset \mathbb{A}^{n+1}$$ is given by the vector space dimension of $$
 k[x_0, \dots , x_n,T]/(P_1 +\cdots +P_{i}+L_1+\cdots +L_{n-i}+L_A+S).
 $$
\end{proof}

Applying Theorem \ref{projective_deg_theorem}, we immediately obtain Algorithm \ref{algorithm:csm_polar}, which allows us to compute the projective degrees of a map $\phi$ defined by a homogeneous ideal $I$ in $k[x_0,\dots, x_n]$.

Using Theorem \ref{projective_deg_theorem}, in the form of Algorithm \ref{algorithm:csm_polar}, and Proposition \ref{propn:aluffi_segre} we may compute the Segre class of a scheme $V$ in $\pp^n$ defined by an ideal $I=(f_0,\dots , f_m)$ in $k[x_0,\dots ,x_n]$ as follows. Assume, without loss of generality, that all generators of $I$ have degree $d$. Applying Proposition \ref{propn:aluffi_segre}, the projective degrees of the map $$
\phi : \mapdefdot{\pp^n}{ \pp^m}{p}{(f_0(p): \cdots : f_m(p))}
$$ can be used to compute the Segre class of the scheme defined by the ideal $I$ in $\pp^n$. Written explicitly in this case, the result of Proposition \ref{propn:aluffi_segre} becomes \begin{equation}
s(V,\pp^n)=1-\sum_{i=0}^n\frac{g_i h^i}{(1+dh)^{i+1}} \in A_*(\pp^n),
\label{eq:Segre_gs}
\end{equation} where $ V=V(I)$, $d=\deg(f_i)$ and $(g_0,\dots,g_n)$ are the projective degrees of the map $\phi$. We summarize this method for computing the Segre class in Algorithm \ref{algorithm:SegreAlg}.

\renewcommand{\figurename}{\textit{Algorithm}}
If we take $\phi$ in Theorem \ref{projective_deg_theorem} above to be the polar map $\varphi$ (see (\ref{eq:polar_map_def})) we have the following corollary, which will allow us to compute the Chern-Schwartz-MacPherson class and Euler characteristic of projective varieties. 

\begin{corr}

Let $V$ be a hypersurface of $\pp^n$ defined by the homogeneous polynomial ideal $(f) $ in $k[x_0, \dots , x_n]$. Since we take the $c_{SM}$ class of $V$ to be the $c_{SM}$ class of its support, i.e. $c_{SM}(V)=c_{SM}(V_{red})$, we assume without loss of generality that $f$ is square-free.  The projective degrees $(g_0,\dots , g_n)$ of $ \varphi : \pp^n  \dashrightarrow \pp^n$,  $$
\varphi:p \mapsto \left( \frac{\partial f}{\partial x_0}(p): \cdots : \frac{\partial f}{\partial x_n}(p) \right),
$$ are given by \begin{equation}
g_i= \dim_k \left( k[x_0, \dots , x_n,T]/(P_1 +\cdots +P_i+L_1+\cdots +L_{n-i}+L_A+S)\right).
\end{equation} Here $P_{\ell},L_{\ell},L_A$ and $S$ are ideals in $R[T]=k[x_0,\dots , x_n,T]$  with $P_{\ell} = \left( \sum_{j=0}^m \lambda_{\ell,j} f_j\right)$ for $ \lambda_{\ell,j}$ a general scalar in $k$, $S =  \left( 1-T \cdot \sum_{j=0}^m \vartheta_{j} f_j \right)$, for $\vartheta_{j}$ a general scalar in $k$, $L_{\ell}$ a general homogeneous linear form for $\ell=1,\dots,n$ and $L_A$ a general affine linear form. Additionally $g_0=1$. \label{corr:polar_deg_csm}
\end{corr}

Corollary \ref{corr:polar_deg_csm} combined with Theorem \ref{Csm_computing _thorem} can be used to compute the Chern-Schwartz-Macpherson Class and Euler characteristic of a projective hypersurface. This formula can be extended to higher codimension using the inclusion/exclusion relation for $c_{SM}$ classes, see Proposition \ref{propn:csm_higher_codim}. This is described explicitly in Algorithm \ref{algorithm:CSMAlg}.

\section{The Algorithms}\label{algorithms}
In this section the result of Theorem \ref{projective_deg_theorem} is used to construct an algorithm to compute the projective degrees using a computer algebra system, presented in Algorithm \ref{algorithm:csm_polar}. Algorithm \ref{algorithm:csm_polar} is in turn used to construct Algorithm \ref{algorithm:SegreAlg} which computes the Segre class $s(V,\pp^n)$ of a subscheme $V$ of $\pp^n$ and Algorithm \ref{algorithm:CSMAlg} which computes $c_{SM}(V)$ and/or $\chi(V)$.

Below we describe Algorithm \ref{algorithm:csm_polar} which applies the result of Theorem \ref{projective_deg_theorem} to compute the projective degrees of a map $\phi:\pp^n \dashrightarrow \pp^m$, $\phi:p \mapsto \left( f_0(p): \cdots : f_m(p) \right) $ corresponding to an ideal $I=(f_0, \dots, f_m)$ of $k[x_0, \dots, x_n]$. R.ideal($f_0,\dots , f_r$) denotes a function which creates the ideal $(f_0,\dots , f_r)$ in the ring $R$ and $k$.random() denotes the function which generates a general element of a field $k$.

\begin{algorithm}\label{algorithm:csm_polar}
\textbf{def projective\_deg\_map:}
\begin{itemize}
\item \textbf{Input:} $I=(f_0, \dots, f_m) $ a homogeneous ideal in $k[x_0,\dots,x_n]$, such that $\deg(f_i)=d$ for all $f_i\neq 0$. 
\item \textbf{Output:} The projective degrees $(g_0,\dots, g_n)$ of a map $ \phi : \pp^n  \dashrightarrow \pp^m$, $\phi:p \mapsto \left( f_0(p): \cdots : f_m(p) \right).$ 
\begin{itemize}
\item Set $ R=k[x_0,\dots ,x_n,T]$. 
\item \textbf{For} $i=0$ \textbf{to} $n$\textbf{:} \begin{itemize}
\item $P= \sum_{\ell=1}^i R.\mathrm{ideal}\left(  \sum \limits_{j=0}^m  k.\mathrm{random()}  \cdot f_j \right)$.
\item $L=\sum_{\ell=1}^{n-i} R.\mathrm{ideal}\left(  \sum \limits_{j=0}^n  k.\mathrm{random()}  \cdot  x_{j} \right)$.
\item $L_A= R.\mathrm{ideal}\left(1+  \sum \limits_{j=0}^n  k.\mathrm{random()}  \cdot  x_{j} \right)$.
\item $V_S=R.\mathrm{ideal}\left(1-T  \sum \limits_{j=0}^m  k.\mathrm{random()}  \cdot  f_j\right)$.
\item $        \mathrm{zero\_dim\_ideal}=P+L+L_A+V_S \subset R$. 
  
\item $    g_i=\dim_k(k[x_0, \dots , x_n,T] / \mathrm{zero\_dim\_ideal})$.
\end{itemize}
\item \textbf{Return} $(g_0,\dots, g_n)$.
\end{itemize}
\end{itemize}
\end{algorithm}

Below we describe Algorithm \ref{algorithm:SegreAlg} which compute the Segre class $s(V,\pp^n)$ in $A_*(\pp^n)$ for $V$ a subscheme of $\pp^n$ defined by a homogenous ideal $I$. We assume, without loss of generality since $I$ is homogenous, that all generators of $I$ have the same degree.
\begin{algorithm}\label{algorithm:SegreAlg}
\textbf{def segre\textunderscore proj\textunderscore deg:}
 \begin{itemize}
\item \textbf{Input:} A homogeneous ideal $I=(w_0, \dots , w_m)$ in $k[x_0,\dots,x_n]$ such that $\deg(w_j)=dj$ for all $j$ defining a scheme $V=V(I)$ in $\pp^n$. 
\item \textbf{Output:} The Segre class $s(V,\pp^n)$ in $A_*(\pp^n)\cong \ZZ[h]/(h^{n+1})$.
\begin{itemize}
\item Compute $(g_0,\dots,g_n)\;= $ projective\_deg\_map$(I)$ (i.e.\ calculate $(g_0,\dots,g_n)$ using Algorithm \ref{algorithm:csm_polar} above).
\item Compute $s(V,\pp^n)=1-\sum_{i=0}^n\frac{g_i h^i}{(1+dh)^{i+1}},$ see \eqref{eq:Segre_gs}.
\item \textbf{return} $s(V,\pp^n)$.
\end{itemize}
\end{itemize}
\end{algorithm}

Below we describe Algorithm \ref{algorithm:CSMAlg} which computes the Chern-Schwartz-Macpherson class $c_{SM}(V)$ in $A_*(\pp^n)$ and/or the Euler chacteristic $\chi(V)$ for $V$ a subscheme of $\pp^n$ defined by a homogenous ideal $I$. $\mathcal{L}_I\mathrm{.parity}(f)$ denotes a function such that $\mathcal{L}_I\mathrm{.parity}(f)=1$ if $f$ is a product of an odd number of generators of $I$ and $\mathcal{L}_I\mathrm{.parity}(f)=-1$ if $f$ is a product of an even number of generators of $I$.

\begin{algorithm}\label{algorithm:CSMAlg}
\textbf{def csm\_polar:}
 \begin{itemize}
\item \textbf{Input:} A homogeneous ideal $I=(f_0,\dots,f_r)$ in $k[x_0,\dots,x_n]$ defining a scheme $V=V(I)$ in $\pp^n$. 
\item \textbf{Output:} $c_{SM}(V)$ in $A_*(\pp^n)\cong \ZZ[h]/(h^{n+1})$ and/or the integer $\chi(V)$.
\begin{itemize}
\item Make a list $\mathcal{L}_I$ of all generators and all products of generators of the ideal $I$.
\item \textbf{For} $f$ \textbf{in} $\mathcal{L}_I$: %compute $c_{SM}(V(f))$.
\begin{itemize}
\item Set $ J=\left(\frac{\partial f}{\partial x_0}, \dots, \frac{\partial f}{\partial x_n} \right)$.
\item Compute the projective degrees $$(g_0,\dots,g_n)= \mathrm{projective\_deg\_map}(J) \;\;\; [\mathrm{See} \;\mathrm{Algorithm \; \ref{algorithm:csm_polar}}].$$
\item Compute $ c_{SM}(V(f))=(1+h)^{n+1}-\sum_{j=0}^{n}g_j (-h)^j(1+h)^{n-j},$ see Theorem \ref{Csm_computing _thorem}.
\item Store $c_{SM}(V(f))$.
\end{itemize}
\item Apply the inclusion/exclusion property of $c_{SM}$ classes (Proposition \ref{propn:csm_higher_codim}) to obtain $$c_{SM}(V)= \sum_{f \in \mathcal{L}_I} \mathcal{L}_I\mathrm{.parity}(f) \cdot c_{SM} \left(V( F_{\left\lbrace S \right\rbrace } )\right)$$ 
\item \textbf{Return} $c_{SM}(V)$ and/or $\chi(V) = \int c_{SM}(V)$. 
\end{itemize}
\end{itemize}
\end{algorithm}

We now give an example of how Algorithms \ref{algorithm:csm_polar}, \ref{algorithm:SegreAlg} and \ref{algorithm:CSMAlg} are used to compute Segre classes, $c_{SM}$ classes and the Euler characteristic. We will again use the variety considered in Example \ref{example:csm_ex}.
\begin{example}
Let $V=V(I)$ be the subvariety of $\pp^4$ defined by the ideal $I=(4x_3x_2x_4x_1-x_0^3x_1,x_0x_1x_3x_4-x_2^3x_3)=(f_0,f_1)$ in $k[x_0,x_1,x_2,x_3,x_4]$. Also set $d=\deg(f_0)=\deg(f_1)=4$.

%Recall that we may write the Chow ring of $\pp^n$ as $A_*(\pp^n)\cong \ZZ[h]/(h^{n+1})$ where $h$ is the rational equivalence class of a hyperplane, meaning a hypersurface $W$ of degree $d$ in $\pp^n$ is represented as $[W]=d\cdot h$ in $A_*(\pp^n) $. 

We first compute the Segre class $s(V,\pp^4)$ of $V$ in $\pp^4$ considered as an element of $A_*(\pp^4) \cong \ZZ[h]/(h^{5})$ where $h$ is the rational equivalence class of a hyperplane, meaning a hypersurface $W$ of degree $d$ in $\pp^4$ is represented as $[W]=d\cdot h$. We will follow the procedure of Algorithm \ref{algorithm:SegreAlg}. Consider the rational map $\phi: \pp^4 \dashrightarrow \pp^1 $ defined by the ideal $I$, that is $$
\phi: p \mapsto (f_0(p):f_1(p)). 
 $$We may compute the projective degrees $(g_0,g_1,g_2,g_3, g_4)$ of this rational map (see \eqref{eq:harris_proj_deg_def}) using Theorem \ref{projective_deg_theorem}.  Let $R=k[x_0,x_1,x_2,x_3,x_4,T]$. Theorem \ref{projective_deg_theorem} gives us that $g_0=1$ and that we may compute$$
g_1= \dim_k(R/(P_1+L_1+L_2+L_3+L_A+S))
$$ where $P_1=(7f_0+9f_1)$ is the ideal in $R$ defined by a general linear combination of the generators of $I$; $L_1=(-11x_0+21x_1-3x_2-18x_3+22x_4)$, $L_2=(31x_0-23x_1+2x_2+47x_3-43x_4)$ and $L_3=(13x_0-52x_1-29x_2+71x_3-15x_4)$ are ideals in $R$ defined by general homogeneous linear forms in $k[x_0,x_1,x_2,x_3,x_4] $,  $L_A=(17-14x_0+41x_1+12x_2-91x_3-3x_4)$ is an ideal in $R$ defined by an affine general linear form in $k[x_0,x_1,x_2,x_3,x_4] $, and $S$ is the ideal of $R$ given by $S=(1-T(3f_0-5f_1))$. The expression $3f_0-5f_1$ in the definition of $S$ is a general linear combination of the generators of $I$. This gives $g_1=4$. In a similar manner we may compute the remaining projective degrees to obtain $$(g_0,g_1,g_2,g_3, g_4)=(1,4,0,0,0).$$ Applying the formula in \eqref{eq:Segre_gs} we obtain \begin{align*}
s(V,\pp^n)=&1- \frac{ 1}{1+4h} -\frac{ 4h}{(1+4h)^{2}}\\
=&768h^4  - 128h^3  + 16h^2\in A_*(\pp^4).
\end{align*}

Now we compute $c_{SM}(V)$ and $\chi(V)$ using the procedure of Algorithm \ref{algorithm:CSMAlg}. By the inclusion/exclusion property of $c_{SM}$ classes (Proposition \ref{propn:csm_higher_codim}) we have that \begin{equation}
c_{SM}(V)= c_{SM}(V(f_0))+c_{SM}(V(f_1))-c_{SM}(V(f_0\cdot f_1)).
\end{equation} We first calculate $ c_{SM}(V(f_0))$; we begin by finding the projective degrees of the map corresponding to the ideal $J$ generated by the partial derivatives of $f_0$ $$J=(\nabla  f_0)=(3x_0^2x_1,-x_0^3+4x_2x_3x_4,4x_1x_3x_4,4x_1x_2x_4,4x_1x_2x_3),$$ that is we must find the projective degrees $(g_0, g_1,g_2,g_3, g_4)$ of the rational map $\varphi: \pp^4 \dashrightarrow \pp^4$ (sometimes referred to as the polar or gradient map \eqref{eq:polar_map_def}) given by $$
\varphi: (p_0:p_1:p_2 : p_3:p_4) \mapsto (3p_0^2p_1:-p_0^3+4p_2p_3p_4:4p_1p_3p_4:4p_1p_2p_4:4p_1p_2p_3).
$$ Using Corollary \ref{corr:polar_deg_csm} we compute that the projective degrees are $(g_0,g_1,g_2,g_3,g_4)=(1,3,6,6,2)$. By Theorem \ref{Csm_computing _thorem} this gives us that \begin{align*}
c_{SM}(V(f_0))=&(1+h)^{5}-\sum_{j=0}^{4}g_j (-h)^j(1+h)^{4-j}\\
=& 5h^4  + 9h^3  + 7h^2  + 4h \in A_*(\pp^4).
\end{align*} 

Similarly we find that the projective degrees of the polar maps corresponding to $f_1$,  and $ f_0f_1$ are $(1, 3, 6, 6, 2)$ and $ (1, 7, 23, 29, 12)$ respectively. This gives the $c_{SM}$ classes: \begin{align*}
c_{SM}(V(f_1))=&5h^4  + 9h^3  + 7h^2  + 4h, \\
c_{SM}(V(f_0f_1))=&5h^4  + 10h^3  + 2h^2  + 8h.
\end{align*}

Combining these we obtain $$c_{SM}(V)=5h^4  + 8h^3  + 12h^2 \in A_*(\pp^4) \cong \ZZ[h]/(h^5).$$ From this we may immediatly conclude that the Euler characteristic of $V$ is $5$ since the Euler chacteristic of $V$ is the degree of the zero dimensional component of $c_{SM}(V)$, i.e.\ the coefficent of $h^4$ in $c_{SM}(V)$ since $V\subset \pp^4$. Eqivilently we may write \begin{align*}
\chi(V)=&\int c_{SM}(V) \\
=&\int 5h^4  + 8h^3  + 12h^2 =5,
\end{align*} where $\int$ denotes the degree of the zero dimensional component of a Chow ring element, that is the coefficent of $h^4$ in this example. \label{example:alg_ex}
\end{example}

 \section{Performance}
 \label{section:Performance}
 In this section we compare the performance of our algorithms to compute Segre classes, $c_{SM}$ classes and Euler to other existing algorithms.  All algorithms are implemented in Macaulay2 \cite{M2} to offer a fair comparison for testing purposes. The Macaulay2 \cite{M2} implementations use Bertini \cite{Bertini} for numerical computations when a numeric option is provided. The methods segre\_proj\_deg (Algorithm \ref{algorithm:SegreAlg}) and csm\_polar (Algorithm \ref{algorithm:CSMAlg}) are also implemented in Sage \cite{sage} and timings for the Sage implementation of csm\_polar (Algorithm \ref{algorithm:CSMAlg}) are included in Table \ref{table:results}. The Sage implementation of our algorithm uses PHCpack \cite{verschelde1999algorithm} for the numerical computation option. 

A list of all examples used for testing benchmarks in Table \ref{table:SegreResults} and Table \ref{table:results} can be found below in Appendix \ref{App:Appendix}. The examples are given in the form of Macaulay2 \cite{M2} input.  

The Macaulay2 \cite{M2} and Sage \cite{sage} implementations of our algorithm for computing $c_{SM}$ classes, Euler characteristics and Segre classes of projective varieties can be found at \url{https://github.com/Martin-Helmer/char-class-calc}. The Macaulay2 \cite{M2} implementation is also available as part of the ``CharacteristicClasses'' package in Macaulay2 version 1.7 and above and can be accessed using the option ``Algorithm$=>$ProjectiveDegree", see the Macually2 documentation \url{http://www.math.uiuc.edu/Macaulay2/doc/Macaulay2-1.7/share/doc/Macaulay2/CharacteristicClasses/html/} for further details. 

Segre (Aluffi) and CSM (Aluffi) refer to the algorithms of Aluffi \cite{aluffi2003computing}, as implemented by Aluffi in the Macaulay2 program available from Aluffi's webpage, \url{http://www.math.fsu.edu/~aluffi/CSM/CSM.html}. The main computational step in the both algorithms of Aluffi is the computation of the Rees algebra. Specifically to calculate $ s(V,\pp^n)$ Allufi computes $Bl_V\pp^n$  and to calculate $c_{SM}(V)$ Aluffi computes $Bl_Y\pp^n$ for $Y$ the singularity subscheme of each hypersurface appearing in Proposition \ref{propn:csm_higher_codim}. 

The algorithm segreClass (E.J.P.) is the algorithm based on Proposition \ref{propn:JostSegre} given by Eklund, Jost and Peterson in \cite{Jost}. CSM (Jost) is the algorithm described in \cite{jost2013algorithm}. For testing of both segreClass (E.J.P.) and CSM (Jost) we used the implementation of Jost available in the ``CharacteristicClasses'' Macaulay2 package on the webpage \url{http://www.math.illinois.edu/Macaulay2/doc/Macaulay2-1.6/share/doc/Macaulay2/CharacteristicClasses/html/}. In Macaulay2 version 1.7 and above Jost's implementations are accessed using the option ``Algorithm$=>$ResidualSymbolic". The main computational step for the algorithms of both \cite{Jost} and \cite{jost2013algorithm} is the computation of the saturations $ J:I^{\infty}$ to compute the residuals as in ($\ref{eq:JostSegre}$). Specifically to calculate $ s(V,\pp^n)$ Jost's implementation computes the residuals via saturations as described in Proposition \ref{propn:JostSegre} and to calculate $c_{SM}(V)$ the implementation computes $s(Y,\pp^n)$ in the same way for $Y$ the singularity subscheme of each hypersurface appearing in Proposition \ref{propn:csm_higher_codim}. % to yield $s(Y,\pp^n),$ which is then used to compute the $c_{SM}$ class.

The method segre\textunderscore proj\_deg uses Algorithm \ref{algorithm:SegreAlg}. Algorithm \ref{algorithm:CSMAlg} is referred to as csm\_polar in Table \ref{table:results}; the Macaulay2 implementation is referred to as csm\_polar (M2) and the Sage implementation is csm\_polar (Sage).  The primary computational cost of Algorithm \ref{algorithm:SegreAlg} and Algorithm \ref{algorithm:CSMAlg} is the computation of the projective degrees $(g_0\dots,g_n)$ which is done by computing the vector space dimension of a ring modulo a zero dimension ideal. This computation can be done symbolically using Gr\"obner bases or numerically using Bertini \cite{Bertini} or some other package for homotopy continuation. 

All symbolic computations are performed over the finite field with $32749$ elements, the numeric computations are done over $\mathbb{Q}$. Note that the $c_{SM}$ class is, technically, only defined when working over fields of characteristic zero (see, for example, \cite{aluffi2limits} for further discussion), however since the result of the computation is the same when working over $\mathbb{Q}$ and over a finite field for a large prime on all examples considered we give the run times over the finite field with $32749$ elements for symbolic computations.  This approach is also used for example computations of characteristic classes by Aluffi \cite{aluffi2003computing} and Jost \cite{jost2013algorithm}, as well as by  Eklund, Jost and Peterson \cite{Jost}. We also note that even when the symbolic methods are run over $\mathbb{Q}$ they still perform better than the numeric versions for each algorithm. All computations were performed on a computer with a 2.40GHz Intel Core i5-450M CPU and 4 GB of RAM.

 We would also like to remark that in the process of developing Algorithm \ref{algorithm:csm_polar} we considered other methods to remove the points in $V(f_1,\dots, f_n)$ (see Theorem \ref{projective_deg_theorem}) which involved performing primary decompositions and evaluating at points in $V(f_1,\dots, f_n)$. However, the main speed up over the algorithm of \cite{Jost} and over the direct numeric calculations was achieved by structuring the equations as they are given in Theorem \ref{projective_deg_theorem}, i.e. by adding the ideal $$S =  \left( 1-T \cdot \sum_{j=0}^m \vartheta_{j} f_j \right), \;\;\; \vartheta_{j} \mathrm{\; a \; general \; scalar \; in \;}k, 
 $$ and working in $k[x_0,\dots,x_n,T]$.
 
The algorithms of Eklund, Jost and Peterson \cite{Jost} and Jost \cite{jost2013algorithm} consider similar algebraic objects (namely the degrees of the residual sets, see Proposition \ref{propn:JostSegre}) to those used in the calculation of the projective degrees in Algorithm \ref{algorithm:csm_polar}. As such it is likely that the performance of the algorithms of \cite{Jost} and \cite{jost2013algorithm} could also be improved by structuring the equations of the residuals considered in \cite{Jost} in the same way as we do here to compute the projective degrees using Theorem \ref{projective_deg_theorem}.   
 \subsection{Timings for the Compution of Segre Classes} \label{segre_results_section}
 In Table \ref{table:SegreResults} we compare the running times of the Segre class computation method using Algorithm \ref{algorithm:SegreAlg} with the running times of two other algorithms to compute Segre classes. 
 
The method of Algorithm \ref{algorithm:SegreAlg} and that of Eklund, Jost and Peterson \cite{Jost} also have numeric implementations, which use the program Bertini \cite{Bertini} for homotopy continuation. However, the numeric implementations of both algorithms are significantly slower than the corresponding symbolic implementations. Only one example in Table \ref{table:SegreResults} finished running in the allotted time (this is the rational normal curve in $\pp^7$); the numeric timings are listed in brackets for this case.

\begin{table}
\resizebox{\linewidth}{!}{
\begin{tabular}{@{} l *6c @{}}
\toprule 
 \multicolumn{1}{c}{{\color{Ftitle} \textbf{Input}}}    & {\color{Ftitle} Segre (Aluffi \cite{aluffi2003computing}) }  &  {\color{Ftitle} segreClass(E.J.P. \cite{Jost})}  &  {\color{Ftitle} segre\textunderscore proj\textunderscore deg (Alg. \ref{algorithm:SegreAlg}) }  \\ 
 \midrule 
  Rational normal curve in $\pp^7$ & - & 7s (9s) & {\color{line}8s} (15s)\\
  Segre embedding of $\pp^2 \times \pp^3$ in $\pp^{11}$  & 2s & -  & {\color{line}52s }\\
   Smooth deg. $81$ variety in $\pp^7$ & - & 36.4s  & {\color{line}1.5s  } \\
 Degree $10$ variety in $\pp^8$& - & 59s  & {\color{line}18s } \\
 Degree $21$ variety in $\pp^9$  & 0.5 s & 33s  & {\color{line}10s } \\
  Degree $48$ variety in $\pp^6$  & - & 173s  & {\color{line}6s }\\
  
\bottomrule
 \end{tabular}}
 \caption{Run time comparision of different algorithms for computing the Segre class of a projective variety. Timings for a numerical implemention of the algorithms using Bertini \cite{Bertini} are included in brackets where available. We use - to denote computations that were stopped after ten minutes (600 s), for the numeric computations that do not finish in less than ten minutes we simply omit the result. \label{table:SegreResults}}
 \end{table}
We note that for all examples except the degree $21$ variety in $\pp^9$ and the Segre embedding of $\pp^2 \times \pp^3$ in $\pp^{11}$ our algorithm performs favourably in comparison to the other algorithms. For these two examples it seems that the particular structure of the ideals being considered happens to favour the computation of the Rees algebra. These examples were included to show that even though Algorithm \ref{algorithm:SegreAlg} tends to be faster in general there are still some cases where the special structure of the ideal being considered makes another technique, such as computing the Rees algebra, more advantageous. Such outliers are less likely to turn up in the $c_{SM}$ class computations since for any codimension greater than one we must compute many $c_{SM}$ classes of different ideals arising from the inclusion/exclusion, and hence the special structure of any one particular ideal plays less of a role. 
 
\subsection{Timings for the Compution of $c_{SM}$ Classes and Euler Characteristics}\label{csm_results_section}

In Table \ref{table:results} we compare the running times of our algorithm to compute the $c_{SM}$ class and Euler characteristic (Algorithm \ref{algorithm:CSMAlg}) with the running times of several other algorithms to compute the $c_{SM}$ class and Euler characteristic. 

The function euler in Table \ref{table:results} is the built in Macaulay2 function which calculates Hodge numbers to compute the Euler characteristic, and does not compute the $c_{SM}$ class. The method euler only works for smooth projective varieties. Note that the Hodge numbers are found by computing the ranks of appropriate cohomology rings and this process is computationally expensive in general; this is likely the reason that the euler function does not perform well for examples in larger ambient dimension and with larger degree. 

We observe that the symbolic implementation of the algorithm described in Algorithm \ref{algorithm:CSMAlg} performs better than the other existing algorithms in all cases shown in Table \ref{table:results}. It is perhaps not surprising that the algorithm of Aluffi \cite{aluffi2003computing} takes longer than the others in many cases as it computes the Rees algebra for each hypersurface, which is in general rather difficult. The algorithm of Jost \cite{jost2013algorithm} computes the Segre class explicitly, using saturations to find the residuals, before computing the $c_{SM}$ class. This also seems to be slower in general than the projective degree calculations of Algorithm \ref{algorithm:CSMAlg}. 

We observe that the numeric implementations of the algorithm of Jost \cite{jost2013algorithm} and csm\_polar are slower than their symbolic counterparts in all tested cases, with the majority not finishing in the allotted time of ten minutes. As was the case with the Segre class computations the symbolic implementation of each algorithm tends to be much faster regardless of which algorithm or which numerical package is used. The reason for the consistently superior performance of the symbolic methods for the types of equations considered in these characteristics class computations is not clear to us. We do, however, believe that the numeric implementations could still be useful for computation both now and in the future, as they are easily parallelizable and their effectiveness on these types of systems could improve over time. 

\begin{table}
\resizebox{\linewidth}{!}{
\begin{tabular}{@{} l *5c @{}}
\toprule 
 \multicolumn{1}{c}{{\color{Ftitle} \textbf{Input}}}    & {\color{Ftitle} CSM (Aluffi)}  &  {\color{Ftitle} CSM (Jost)}  &  {\color{Ftitle} csm\textunderscore polar (M2)}&  {\color{Ftitle} csm\textunderscore polar (Sage)}& {\color{Ftitle} euler } \\ 
 \midrule 
  Twisted cubic & 0.3s & 0.1s (35s) & {\color{line} 0.1s} (37s) & {\color{line}0.1s} (0.6s) & 0.2s \\ 
 Segre embedding of $ \pp^1 \times \pp^2 $ in $\pp^5$ & 0.4s & 0.8s (148s) & {\color{line}0.2s }(152s)&  {\color{line}0.2s} (57s)& 0.2s \\
 Smooth degree $8$ variety in $\pp^4$ & - & 1.2s (-) & {\color{line} 0.6s } (-)& {\color{line}0.2s} (28s)& 20.1s \\
 Smooth degree $4$ variety in $\pp^{10}$ & - & 56.8s  & {\color{line}6.4s }& {\color{line}2.2s}& - \\
 Smooth degree $6$ variety in $\pp^{7}$ & - & -  & {\color{line}148.5s } & {\color{line}77.7s}& - \\
 Deg. $12$ hypersurface in $\pp^3$ & 25.3s & 1.0s  & {\color{line}0.2s} &{\color{line} 0.1s}& n/a \\
 Degree $3$ variety in $\pp^8$ & - & 85.2s  & {\color{line}2.2s }& {\color{line}1.0s}& n/a \\
 Degree $5$ variety in $\pp^{10}$ & - & -  & {\color{line}10s } & {\color{line}2.3s}& n/a \\
 Degree $16$ variety in $\pp^5$& - & -  & {\color{line} 1.3s }& {\color{line}0.3s}& n/a \\
\bottomrule
 \end{tabular}}
 \caption{Comparison of Algorithm \ref{algorithm:CSMAlg} (csm\_polar) with different known algorithms to compute the $c_{SM}$ class and Euler characteristic of a projective variety. The - denotes a computation that did not finish after running for ten minutes (600s), n/a indicates the variety is singular and hence the algorithm euler is not applicable. Numeric timings are given in brackets (-) where available, numeric computations taking longer than 600s are omitted. 
 \label{table:results}}
 \end{table}

We believe that given the favourable performance of Algorithm \ref{algorithm:SegreAlg} and Algorithm \ref{algorithm:CSMAlg} on a wide variety of examples we may conclude that these methods provide a useful complement to the existing methods which compute Segre and $c_{SM}$ classes and the Euler characteristic for subschemes of projective space.

\normalsize
 
 \section*{Acknowledgments}
This research was partially supported by the Natural Sciences and Engineering
Research Council of Canada (NSERC). Also the author would like to thank Eric Schost for many helpful discussions on the content of this note.

 \newpage

\appendix
\section{Appendix} \label{App:Appendix}
In this appendix we give the examples used for testing in Tables \ref{table:SegreResults} and \ref{table:results}. The examples are given in the form of plain text Macaulay2 \cite{M2} code. We assume that the function to compute Segre classes is named Segre and the function to compute $c_{SM}$ classes is named CSM. This is the convention used in our M2 package ``CharClassCalc" available at \url{https://github.com/Martin-Helmer/char-class-calc}.

Below are the examples listed in Table \ref{table:SegreResults} which are used for testing the performance of Algorithm \ref{algorithm:SegreAlg}, our algorithm for computing the Segre class of a projective variety. 
\scriptsize \vspace{-5mm}
\begin{verbatim}
----------------------------------------------------
--Segre Examples
----------------------------------------------------
needsPackage "CharClassCalc"
TEST ///
--Rational Normal curve in P^7
    n=7; R=ZZ/32749[y_0..y_n];
    M = matrix{{y_0..y_n},{y_1..y_n,y_0}};
    I=minors(2,M);
    time Segre I
///
TEST ///
--Segre embedding of P^2xP^3 in P^11
    n=11; R=ZZ/32749[x_0..x_n];
    M = matrix{{x_0,x_1,x_2,x_3},{x_4,x_5,x_6,x_7},{x_8,x_9,x_10,x_11}};
    I=minors(2,M);
    time Segre(I)
///
TEST ///
--Smooth degree 81 variety in P^7
    n=7; R=ZZ/32749[y_0..y_n];
    I = ideal(random(3,R),random(3,R),random(3,R),random(3,R));
    time Segre(I) 
///
TEST ///
--Degree 10 variety in P^8
    n=8; R=ZZ/32749[y_0..y_n];
    M = matrix{{random(1,R),random(1,R),random(1,R)},
    {random(1,R),random(1,R),random(1,R)},
    {random(1,R),random(1,R),random(1,R)},
    {random(1,R),random(1,R),random(1,R)}};
    I=minors(2,M);
    time Segre(I)
///
TEST ///
--Degree 21 variety in P^9 
    n=9; R=ZZ/32749[x_0..x_n];
    I=ideal((4*x_3*x_2*x_4-x_0^3)*x_1^3,x_5^3*(x_0*x_1*x_4-x_2^3),
    x_9^3*(x_7*x_8*x_6-x_4^3)-x_7^5*x_0,
    7*x_1^3*(x_2*x_1*x_6-x_9^3)-3*x_2^3*x_0^3);
    time Segre(I)
///
TEST ///
--Degree 48 variety in P^6 
    n=6; R=ZZ/32749[x_0..x_n];
    M = matrix{{x_1*x_4^2*x_3-x_1^4,random(4,R),random(4,R)},
    {random(4,R),x_0^2*x_5^2-x_6^3*x_0,random(4,R)}};
    I=minors(2,M);
    time Segre(I)
///    
\end{verbatim} 
\normalsize Below are the examples listed in Table \ref{table:results} which are used for testing the performance of Algorithm \ref{algorithm:CSMAlg}, our algorithm for computing the $c_{SM}$ class of a projective variety.  \scriptsize
\begin{verbatim}
----------------------------------------------------
--CSM Examples
----------------------------------------------------
TEST ///
--Twisted Cubic
    n=3; R=ZZ/32749[x_0..x_n];
    K=ideal(x_1*x_3-x_2^2, x_2*x_0-x_3^2,x_1*x_0-x_2*x_3)
    time CSM K
///
TEST ///
--Segre embedding of P^1xP^2 in P^5
    n=5; R=ZZ/32749[y_0..y_n];
    I=ideal(y_0*y_4-y_1*y_3,y_0*y_5-y_2*y_3,y_1*y_5-y_4*y_2);
    time CSM I    
///
TEST ///
--Smooth degree 8 variety in P^4
    n=4; R=ZZ/32749[z_0..z_n];
    K=ideal(-11796*z_0^2 + 2701*z_0*z_1 + 10725*z_1^2 - 11900*z_0*z_2 - 
    11598*z_1*z_2+ 11286*z_2^2 + 5210*z_0*z_3 - 7485*z_1*z_3 + 11208*z_2*z_3 
    + 5247*z_3^2 -4745*z_0*z_4 - 15915*z_1*z_4 + 14229*z_2*z_4 - 11236*z_3*z_4 + 
    10583*z_4^2, 6934*z_0^2 + 1767*z_0*z_1 + 9604*z_1^2 - 4343*z_0*z_2 - 10848*z_1*z_2 -
    16357*z_2^2 + 8747*z_0*z_3 - 13140*z_1*z_3 - 7136*z_2*z_3 + 3115*z_3^2 -
    3741*z_0*z_4 + 14969*z_1*z_4 + 10956*z_2*z_4 - 10016*z_3*z_4 + 13449*z_4^2,
    12153*z_0^2 - 4789*z_0*z_1 - 9183*z_1^2 - 15107*z_0*z_2 - 5045*z_1*z_2 +
    6082*z_2^2 - 13665*z_0*z_3 + 4455*z_1*z_3 - 3129*z_2*z_3 + 14146*z_3^2 -
    1424*z_0*z_4 + 11305*z_1*z_4 + 4882*z_2*z_4 - 14665*z_3*z_4 - 10270*z_4^2)
    time CSM(K)
///
TEST ///
--Smooth degree 4 variety in P^10 
   n=10;R=ZZ/32749[x_0..x_n];
   I=ideal(random(2,R),random(2,R));
   time CSM I
///
TEST ///
--Smooth degree 6 variety in P^7
   n=7; R=ZZ/32749[y_0..y_n];
   I=ideal(2*y_0^3+12*y_1^3+96*y_2^3 + 19*y_3^3+12*y_4^3+y_6^3+5*y_7^3, random(2,R));
   time CSM(I)
///
TEST ///
--Degree 12 hypersurface in P^3
   n=3;  R=ZZ/32749[x_0..x_n];
   I=ideal(x_2^6*x_3^6+3*x_1^2*x_2^4*x_3^4*x_0^2+3*x_1^4*x_2^2*x_3^2*x_0^4-3*x_2^4*x_3^4*x_0^4+
   x_1^6*x_0^6+21*x_1^2*x_2^2*x_3^2*x_0^6-3*x_1^4*x_0^8+3*x_2^2*x_3^2*x_0^8+3*x_1^2*x_0^10-x_0^12)
   time CSM(I)
///
TEST ///
--Degree 3 variety in P^8
   n=8; R=ZZ/32749[x_0..x_n];
   M = matrix{{random(1,R),random(1,R),random(1,R)},{random(1,R),random(1,R),random(1,R)}};
   I=minors(2,M);
   time CSM(I)   
///
TEST ///
--Degree 5 variety in P^10  
   n=10; R=ZZ/32749[x_0..x_n];
   M = matrix{{x_0^2-x_1^2,22*x_3-35*x_9-13*x_2,x_9-x_7+5*x_3},
   {x_8+9*x_0+4*x_1,7*x_1-33*x_5+23*x_6,random(1,R)}};
   I=minors(2,M);
   time CSM(I)
///
TEST ///
--Degree 16 variety in P^5
   n=5; R=ZZ/32749[x_0..x_n];
   I=ideal((4*x_3*x_2*x_4-x_0^3)*x_1,x_5*(x_0*x_1*x_4-x_2^3))
   time CSM(I)
///
\end{verbatim}\normalsize
\newpage
% \tiny
%\begin{itemize}
%\item All testing done in Macaulay2.  
%\item CSM (Aluffi) uses blow-ups to calculate the $c_{SM}$ class, \cite{aluffi2003computing}.
%\item CSM (Jost) finds degrees of certain residual sets to calculate the $c_{SM}$ class, \cite{jost2013algorithm}.
%\item csm\textunderscore polar is the algorithm presented here.
%\item euler computes just the Euler characteristic using Hodge numbers, this only works for non-singular varieties.
%\item - means the computation was stopped after $600s $ ($10$ minutes).  
%\end{itemize}
%\normalsize
\newpage
\bibliographystyle{plain}
\bibliography{refs}
\end{document}